\documentclass[final]{siamltex}
\usepackage{amsmath}
\usepackage{mathtools}
\usepackage{mathrsfs}
\usepackage{amssymb}
\usepackage{amsfonts}
\usepackage{amsxtra}
\usepackage{amstext}
\usepackage{amsbsy}
\usepackage{amscd}
\usepackage{graphicx}
\usepackage{float}
\usepackage{cite}
\usepackage{xcolor}
\usepackage{srcltx} 
\usepackage{marginnote,slashbox}
\usepackage{rotating}
\definecolor{darkblue}{rgb}{0.0,0.0,0.6}
\definecolor{darkgreen}{rgb}{0.0,0.6,0.0}
\usepackage[pdftex,colorlinks=true,urlcolor=darkblue,citecolor=darkblue,linkcolor=darkblue]{hyperref}
\usepackage[utf8]{inputenc}      
\usepackage{booktabs}            
\usepackage{url}                 
\usepackage[T1]{fontenc}         
\usepackage{algorithmic}         
\usepackage[pagewise]{lineno}
\usepackage{calc}
\usepackage[notcite,notref]{showkeys}

\numberwithin{table}{section}    
\numberwithin{figure}{section}   
\numberwithin{equation}{section} 
\usepackage{my_latex_commands}
\newtheorem{assumption}[theorem]{Assumption}
\newtheorem{remark}[theorem]{Remark}

\renewcommand{\bl}[1]{\textcolor{black}{#1}}
\newcommand{\ro}[1]{\textcolor{black}{#1}}

\newcommand{\dz}{\delta z}
\newcommand{\ff}{\forall\,}

\newcommand{\eps}{\varepsilon}
\newcommand{\e}{\varepsilon}

\newcommand{\yy}{\bar y}

\newcommand{\ld}{L^2(D)}
\newcommand{\lo}{L^2(D)}
\newcommand{\hde}{H^s(D\setminus \bar E)}
\newcommand{\lde}{L^\infty(D\setminus \bar E)}
\newcommand{\hd}{H^1_0(D)}

\newcommand{\g}{\gamma}

\newcommand{\li}{L^\infty(D)}

\newcommand{\dense}{\overset{d}{\embed}}
\renewcommand{\o}{\omega}

\newcommand{\dn}{D_n^{convex}}
\newcommand{\dnn}{D_n^{concave}}

\begin{document}

\title{Optimal control of  a non-smooth elliptic PDE with  non-linear term acting on the control}
\date{\today}
\author{L.\ Betz\footnotemark[1]}
\renewcommand{\thefootnote}{\fnsymbol{footnote}}
\footnotetext[1]{Faculty of Mathematics, University of W\"urzburg,  Germany}
\renewcommand{\thefootnote}{\arabic{footnote}}

\maketitle

 \begin{abstract}
This paper continues the investigations from \cite{p1} and  is  concerned with the derivation of first-order conditions for a control constrained optimization problem governed by a non-smooth elliptic PDE. The control enters  the state equation not only  linearly but also as the argument of a regularization of the Heaviside function. The non-linearity which acts  on the state  is locally  Lipschitz-continuous and not necessarily differentiable, i.e., non-smooth. This excludes the application of standard adjoint calculus. We derive conditions under which a strong stationary optimality system can be established, i.e., a system that is  equivalent 
 to the purely primal optimality condition saying that the directional derivative of the reduced objective 
 in feasible directions is nonnegative. For this, two assumptions are made on the unknown optimizer.
  Some of the presented findings are employed in the recent contribution \cite{p3}, where    limit optimality systems for non-smooth shape optimization problems \cite{p1} are established.

 \end{abstract}



  \section{Introduction}
This paper is concerned with the derivation of optimality conditions for the following non-smooth optimal  control problem\begin{equation}\tag{$P_\eps$}\label{p10}
 \left.
 \begin{aligned}
  \min_{g \in   \FF} \quad & \int_E (y(x)-y_d(x))^2 \; dx+\alpha\,\int_{D} (1-H_\eps(g))(x)\,dx+\frac{1}{2}\,\|g-\bar g_{\mathfrak{sh}}\|_{\WW}^2  \\   \text{s.t.} \quad & 
  \begin{aligned}[t]
   -\laplace y + \beta(y)+\frac{1}{\eps}H_\eps(g)y&=f +\eps g\quad \text{in }D,
   \\y&=0\quad \text{ on } \partial D, \end{aligned} 
 \end{aligned}
 \quad \right\}
\end{equation}where $\eps>0$ is a fixed parameter. The fixed domain $D\subset \R^N, N \in\{2,3\},$ is bounded, of class $C^{1,1}$, and $f \in L^2(D)$. The symbol $-\laplace: \hd \to H^{-1}(D)$ denotes the Laplace operator in the distributional sense; note that $\hd$  is the closure of the set {$ C^\infty_c(D) $} w.r.t.\,the $H^1(D)-$norm, {while $H^{-1}(D)$ stands for its dual}. The control $g$ enters the state equation not only linearly but also  in a \textit{non-linear} fashion through the mapping $H_\e$. This is the regularization of the Heaviside function, cf.\,\eqref{reg_h} below.  The mapping $\beta$ is \textit{non-smooth}, that is, not necessarily differentiable, see Assumption  \ref{assu:stand}   below for details.
 This constitutes one of the main challenges in the present manuscript, as it  excludes the application of standard adjoint calculus techniques. 
The desired state $y_d$ is an $L^2-$function, which is defined on the  set  $E$ (Assumption \ref{assu:E}).   The parameter $\alpha$ appearing in the objective is supposed to satisfy $\alpha \geq 0$.   In all what follows, $\WW$ is the Hilbert space $\ld \cap \hde$, $s> 0$, endowed with the norm 
\[\|\cdot\|_{\WW}^2:=\|\cdot\|_{\ld}^2+\|\cdot\|_{H^s(D\setminus \bar E)}^2,\]
and 
\begin{equation}\label{f_tilde}
 \FF:=\{g \in \WW: g \leq 0\text{ a.e.\,in }E\}.
\end{equation}
 In the objective of \eqref{p10}, $\bar g_{\mathfrak{sh}} \in \FF$ is fixed.

The present paper is deeply connected to the  contribution \cite{p1}, where we introduced  an approximation scheme of fixed domain type \cite{nt, nt2} for a non-smooth  shape optimization problem. In \cite{p1}, the mapping $\bar g_{\mathfrak{sh}}$ is a so-called locally optimal shape function, which arises as the limit of a sequence of local minima of \eqref{p10} \cite[Thm.\,5.2]{p1}. 
While it continues the investigations from \cite{p1}, the current work stands on its own. Let us underline that  no knowledge on the topic of shape and topology optimization is required here.  This was precisely the idea in \cite{p1}: to introduce an approachable approximating optimal control problem where the admissible set is a (preferably convex) subset of a Hilbert space of functions.   By proceeding in this way, the variable character of the geometry, that is, the main feature and difficulty of the original shape optimization problem \cite{p1}, is out of the picture.  As advertised in \cite{p1}, the present manuscript is entirely concerned with the derivation of optimality conditions for the respective approximating minimization problem, namely for \eqref{p10}.
Despite of the non-smoothness, we establish an optimality system that is equivalent to the first-order necessary optimality condition in primal form, i.e., of strong stationary type. Since we deal with control constraints,  one expects the necessity of so-called "constraint qualifications" \cite{mcrf, wachsm_2014, amo} (assumptions on the unknown local optimizer  \cite[Sec.\ 2]{gk}, abv.\,"CQs"). 
Nevertheless, an additional CQ is required, which is not  related to the constraint in the definition of $\FF$. This is necessary because  the control $g$ appears in the argument of the non-linear mapping $H_\e$ (which then gets multiplied by the state). We can state an estimate featuring  the given data only under which this is satisfied (Corollary \ref{corr}). \ro{If we also take the framework of \cite{p1,p3} into account, the optimality system from Theorem \ref{thm} is strong stationary for $\e$ ``small'', provided that $\beta$ is locally convex around each non-differentiable point that is attained  by the optimal state  (Remark \ref{remm}).} The presence of the $H^s$ term in the objective of \eqref{p10} ensures the existence of optimal solutions and the validity of  a first optimality system. We underline that we  consider the $H^s(D\setminus \bar  E)$ norm (not the full $H^s(D)$ norm), in order to conclude strong stationary optimality conditions for local optima of  \eqref{p10}, see Theorems \ref{thm} and \ref{thm:equiv_B_strong} below. If $\WW=H^s(D)$, the sign conditions \eqref{sign_p0}-\eqref{sign_p1} are not  available, and the optimality system for \eqref{p10} in Theorem \ref{thm} reduces to the weaker one obtained via smoothening methods, cf.\,\eqref{eq:lem} and Remark \ref{rem:sign}.  

\bl{We emphasize that the upcoming analysis can be extended 
to situations where  the assumptions on the involved quantities in \eqref{p10} are  weaker. The mapping $H_\e$ can be replaced by a  Lipschitz continuous, differentiable function. An additional dependence on the spatial variable  can be included in the non-smoothness $\beta$. In this case, $\beta:D \times \R \to \R$ is supposed to be a Carath\'eodory function with $\beta(\cdot,0) \in \ld$, which is \ro{non-decreasing} in its second variable and satisfies a Lipschitz continuity condition similar to the one in Assumption \ref{assu:stand}.\ref{it:stand2}. A (smooth) divergence-gradient-operator can be considered instead of the Laplacian.
The reason why we stick to the present setting is that these specific  requirements are needed to prove that \eqref{p10} is an approximating problem for the shape optimization problem from \cite{p1}, cf.\,\cite[Thm.\,5.2]{p1}.
In the recent manuscript \cite{p3} we pass to the limit $\e \to 0$ in the optimality system associated to \eqref{p10} and we  work again  in the framework of  \cite{p1}, the final goal being to obtain first-order conditions for the original shape optimization problem. Therein, it is   needed that we confine ourselves to additional  requirements such as $s>1$ and $D \subset \R^2$, as these are necessary for \cite[Thm.\,5.2]{p1}.}


%

Deriving necessary optimality conditions is a challenging issue even in finite dimensions, where a special attention is given to MPCCs  (mathematical programs with complementarity constraints).
In \cite{ScheelScholtes2000} a detailed overview of various optimality conditions of different strength was introduced, the  most rigorous concept being strong stationarity. See also \cite{HintermuellerKopacka2008:1} for the infinite dimensional case.
Roughly speaking, the strong stationarity conditions involve an optimality system, which is equivalent to the purely primal conditions saying that the directional derivative of the reduced objective in feasible directions is nonnegative (which is referred to as B-stationarity).
While there are plenty of contributions in the field of optimal control of 
smooth problems \cite{troe}, fewer papers are dealing with 
non-smooth problems. Most of these works resort to smoothening techniques that go back to \cite{barbu84}, see also \cite{tiba} for an extensive study. The optimality systems derived in this way are of intermediate strength and are not expected to be of strong stationary type, since one always loses information when passing to the limit in the regularization scheme (cf.\ also Remark \ref{rem:sign}). Thus, proving strong stationarity for optimal control of  non-smooth problems requires direct approaches, which employ the limited differentiability properties of the control-to-state map. Based on the ideas from the pioneering contribution \cite{mp84} (obstacle problem), the literature on this topic has been growing in the past years and we mention here  \cite{ hmw13,  wachsm_2014, DelosReyes-Meyer, brok_ch, quasi , cc, ira_w, by18}, where strong stationarity for the control of  various kinds of VIs has been established. Regarding strong stationarity for  the optimal control of non-smooth PDEs, we mention  \cite{paper} (parabolic PDE), \cite{cr_meyer} (elliptic PDE), \cite{quasi_nonsmooth} (elliptic quasilinear PDE), \cite{st_coup} (coupled PDE system), and the more recent works \cite{ mcrf, jc_pq, aos, amo}. We point out that  the vast majority of these papers  do not take control constraints into account. This case is more difficult to handle, as it requires assumptions on the unknown optimizer ("constraint qualifications") \cite{wachsm_2014, mcrf, amo}. With the exception of \cite{amo}, we are not aware of any contribution dealing with strong stationarity,  where the control appears in the argument of a non-linear mapping in the state equation, as in \eqref{p10}. Note that this aspect might also require  CQs, see  \cite[Assump.\,4.3, Rem.\,4.12]{amo}, \cite[Rem.\,2.16]{st_coup} and Assumption \ref{assu:h_a} below.

Let us give an overview of the structure and main results of this paper. After introducing the assumptions,   we recall some  results from \cite{p1} that will be useful in the sequel (section \ref{sec:2}). Then, at the beginning of section \ref{oc}, we state a first (rather standard) optimality system  for local optima of \eqref{p10} in Lemma \ref{lem:tool}. This is obtained under the mild requirement that $\beta$ is locally convex or concave around its non-differentiable points (Assumption \ref{assu:reg}).
The optimality conditions from Lemma \ref{lem:tool}  serve as an \textit{auxiliary result only} and are not of strong stationary type (Remark \ref{rem:sign}). They are however important and will be exploited later on in the proof of the main result, i.e., Theorem \ref{thm}. 
In subsection \ref{b} we establish the directional differentiability  of the control to state map (Lemma \ref{lem:dirderiv}), \bl{which allows us to state the first-order necessary condition in its primal form, also known as  B-stationarity, see \eqref{eq:vi}}. Then, in subsection \ref{sec:ss}  we proceed towards the derivation of strong stationary conditions   for the optimal control  problem \eqref{p10}, see Theorems \ref{thm} and \ref{thm:equiv_B_strong}. 

We point out that the formulation of \eqref{p10} is slightly reminiscent of the control problem recently studied in \cite{mcrf} with an \textit{additional challenge}: the control appears in the argument of the non-linearity $H_\eps$, which then gets multiplied by the state. To deal with this aspect, we carefully adapt the ideas from \cite{mcrf} in the proof of Theorem \ref{thm}. \bl{Having established a first optimality system in Lemma \ref{lem:tool}, we just need to prove suitable sign conditions for the adjoint state, see \eqref{needed}. These follow from the B-stationarity result by imposing two CQs. The first one, i.e., Assumption \ref{assu:h_a}, ensures  that the range of the derivative of the whole term acting on $g$ in the state equation, i.e., $\frac{1}{\e}H_\e'( g ) y -\e$, covers an entire space \cite[Rem.\,2.16]{st_coup}. This is  different from the situation in \cite{mcrf}, where there was no coupling of control and state. The control to state map associated to the ``linearized'' state equation in \cite{mcrf} was already surjective.
Hence, Assumption \ref{assu:h_a}  is the price to pay for \ro{obtaining additional conditions to the one from Lemma \ref{lem:tool}} for a non-smooth optimal control problem that features   a \textit{non-linear term acting on the control
   in  the state equation}, see Theorem \ref{thm}.}  In   section \ref{sec:rem}, we derive a simple condition on the given data under which this is true, while keeping in mind that, in \cite{p3}, we are looking at  Assumption \ref{assu:h_a}   in the setting of  \cite{p1} (Lemma \ref{lem}).
\bl{Indeed, \cite{amo} also dealt with a coupling of control and state in the presence of control constraints, but at the same time, the  number of controls was larger than two. While the  constraints were imposed on a certain control, the state was multiplied with the other (unconstrained) controls. For this reason, the above mentioned surjectivity was available in \cite{amo}  without imposing a CQ related to control constraints \cite[Rem.\,4.13]{amo}. In the present paper, we cannot dispense of the latter as we deal with a single control.}  Our  second CQ, i.e., Assumption \ref{assu:cq}, is  a  natural claim  \cite{wachsm_2014, mcrf} and  it is also less restrictive when compared with  findings obtained in other similar contributions (Remark \ref{rem:needed}). In fact, \ro{Assumption \ref{assu:cq} is not even needed in the framework of \cite{p3}, due to an indispensable requirement in \cite{p1}, which yields the desired sign condition  for the adjoint state (for $\eps$ small) when $\beta$ is locally convex around its non-smooth points (Remark \ref{remm}).
\subsection*{Notation}
 Let
$X$ and $Y$ be two Banach spaces. If $X$ is compactly embedded in $Y$ we write $X \embed \embed Y$ and $X \dense Y$ means that $X$ is dense in $Y.$ The open ball around $x\in X$ with radius $M>0$ is denoted by $B_X(x,M)$.
In the sequel, $d(A,B)$ is the distance between two sets $A,B \subset D$, i.e, \[d(A,B)=\inf_{x\in A, y \in B} \|x-y\|_{\R^N}.\]
The symbol $\chi_A$ stands for the characteristic function associated with the set $A$.
\bl{The $N$-dimensional Lebesgue measure of (a measurable set)  $A$ is abbreviated by $\mu(A).$}
}
  \section{Standing assumptions and preliminary results}\label{sec:2}
  Throughout the paper, $\eps>0$ is a fixed parameter. We begin this section by specifying the precise assumptions on the non-smoothness $\beta$ appearing in the state equation in \eqref{p10}. We also mention the requirements needed for the  set $E$. For the rest of the assumptions on the given data, see the introduction.  
\begin{assumption}[The non-smoothness]\label{assu:stand}
 \begin{enumerate}
 \item\label{it:stand2}The  function  $\beta: \R \to \R$ is    
\ro{non-decreasing} and {locally} Lipschitz continuous {in the following sense: F}or all $M>0$, there exists 
  a constant $L_M>0$ such that
  \begin{equation*}
   |\beta(z_1) - \beta(z_2)| \leq L_M |z_1 - z_2| \quad \forall\, z_1, z_2 \in [-M,M] .
  \end{equation*}
  \item\label{it:stand3}The mapping $\beta$ is directionally differentiable at every point, i.e., 
  \begin{equation*}
   \Big|\frac{\beta(z + \tau \,\dz) - \beta(z)}{\tau} - \beta'(z;\dz)\Big| \stackrel{\tau \searrow 0}{\longrightarrow} 0 \quad \forall \, z,\dz \in \R.
  \end{equation*}
  \end{enumerate}
\end{assumption}

By Assumption \ref{assu:stand}, it is straightforward to see that the Nemytskii operator $\beta:\li  \to \li$ is well-defined. Moreover, this is Lipschitz continuous on bounded sets in the following sense:  for every  $M > 0$, 
  there exists $L_M > 0$ so that
  \begin{equation}\label{eq:flip}
   \|\beta(y_1) - \beta(y_2)\|_{L^q(D)} \leq L_M \, \|y_1 - y_2\|_{L^q(D)} \quad \forall\, y_1,y_2 \in \clos{B_{\li}(0,M)},\ \forall\, 1\leq q \leq \infty.
  \end{equation}

\begin{assumption}\label{assu:E}
The  set $E\subset D$  is a  subdomain with boundary of measure zero. Moreover, $d(\clos E,\partial D)>0.$ \end{assumption}

In the rest of the paper, one tacitly supposes that Assumptions \ref{assu:stand} and \ref{assu:E} are always fulfilled without mentioning them every time. 

\begin{definition}
The non-linearity $H_\eps:\R \to [0,1]$ is defined as follows 
 \begin{equation}\label{reg_h}
H_\eps(v):=\left\{\begin{aligned}0,&\quad \text{if }v\leq 0,
\\\frac{v^2(3\eps-2v)}{\eps^3},&\quad \text{if }v\in (0,\eps),
\\1,&\quad \text{if }v\geq \eps.
\end{aligned}\right.\end{equation}\end{definition}

For later purposes, we compute here its derivative:
 \begin{equation}\label{reg_h'}
H_\eps'(v):=\left\{\begin{aligned}0,&\quad \text{if }v\leq 0,
\\\frac{6v(\eps-v)}{\eps^3},&\quad \text{if }v\in (0,\eps),
\\0,&\quad \text{if }v\geq \eps.
\end{aligned}\right.\end{equation}
\begin{remark}\label{rem:h_e}The mapping $H_\eps$ introduced in \eqref{reg_h}  is Lipschitz continuous and continuously differentiable.
It is obtained as a regularization of the Heaviside function $H:\R \to [0,1]$, $H(v)=0$ if $v\leq 0$ and $H(v)=1$ if $v>0.$
Note that Heaviside functions and their regularizations play an essential role  in the context of shape optimization via fixed domain approaches, see for instance \cite{nt, nt2, p1} and the references therein.
\end{remark}

We proceed by recalling some useful findings from \cite{p1}.
\begin{lemma}[Solvability of the state equation, {\cite[Lem.\,4.5]{p1}}]\label{lem:S}
For any right-hand side $g \in L^2(D)$, the state equation in  \eqref{p10} admits a unique solution $ y \in H^1_0(D) \cap H^2(D)$. 
This satisfies 
\begin{equation}\label{est_y}
\|y\|_{H^1_0(D) \cap C(\bar D)} \leq \ro{c\|f-\beta(0)\|_{\ld}+c\,\eps\,\|g\|_{\ld},}
\end{equation}
where $c>0$ is \ro{independent of  $\beta, f, \eps$ and  $g$}.
The control-to-state mapping $S:L^2(D) \to H^1_0(D) \cap H^2(D)$ is  Lipschitz continuous on bounded sets, i.e.,
for every  $M > 0$, 
  there exists $L_{M} > 0$ so that
  \begin{equation}\label{slip}
   \|S(g_1) - S(g_2)\|_{H^1_0(D) \cap H^2(D)} \leq L_{M} \, \|g_1 - g_2\|_{L^2(D)} \quad \forall\, g_1,g_2 \in \clos{B_{\ld}(0,M)}.
  \end{equation}

\end{lemma}

Since  it suffices that $s>0$ to ensure the compactness of the embedding $\hde \embed \embed L^2(D \setminus \bar E)$ 
\cite[Thm.\,4.14]{kaballo}, all the arguments in the proof of {\cite[Lem.\,4.13]{p1}} remain unaffected by the fact that we now work with $s>0,$ not with $s>1$ as in \cite{p1}. Thus, in the present setting, the next two results are valid.

\begin{lemma}[Convergence properties, {\cite[Lem.\,4.13]{p1}}]\label{wc}
Let $\{g_n\} \subset  \FF$ with \[g_n \weakly g \quad \text{in } \WW,\text{ as }n \to \infty.\] Then, 
\begin{equation}\label{he_conv}
H_\eps (g_{n}) \to H_\eps(g) \quad \text{in }L^2(D), \ \text{ as }n \to \infty,\end{equation}
\begin{equation}\label{wc_S}
S(g_n) \weakly S(g) \quad \text{in }H^2(D) \cap \hd, \text{ as }n \to \infty.\end{equation}
\end{lemma}

\begin{proposition}[Existence of solution for \eqref{p10}, {\cite[Prop.\,4.14]{p1}}]\label{ex_e}
The approximating optimal control problem \eqref{p10} admits at least one global minimizer in $\FF$.
\end{proposition}

\section{Optimality conditions for \eqref{p10}}\label{oc}
Throughout this section, we are concerned with the derivation of optimality conditions for a local optimum of \eqref{p10}. This is defined as follows.

\begin{definition}[{\cite[Def.\,5.1]{p1}}]\label{def_e}We say that $\bar g \in \FF$ is locally optimal for \eqref{p10} in the $\ld$ sense, if there exists $r>0$ such that 
\begin{equation}\label{loc_opt_e}
j(\bar g) \leq j(g) \quad \ff g \in \FF \text{ with }\|g-\bar g\|_{\ld}\leq  r,
\end{equation}where $$j(g):=\int_E (S( g)(x)-y_d(x))^2 \; dx+\alpha\,\int_{D} (1-H_\eps(g))(x)\,dx+\frac{1}{2}\,\|g-\bar g_{\mathfrak{sh}}\|_{\WW}^2$$ is the reduced cost functional associated to the control problem \eqref{p10}.
\end{definition}

\begin{remark}
As in \cite{p1}, we deal with local optima in the $\ld$-sense, as this type of local minimizer is the  relevant one for us. In \cite[Thm.\,5.2]{p1} it was shown that the respective locally optimal (shape) function is approximated via a sequence of local optima as in Definition \ref{def_e}. The passage to the limit $\e \to 0$ in the optimality system associated to this kind of local minimizer  is performed in \cite{p3}.
\end{remark}

We begin by stating a first optimality system, which is obtained via classical smoothening techniques that go back to \cite{barbu84}. To do so, we need the following mild requirement on the non-smoothness appearing in the state equation:
\begin{assumption}\label{assu:reg}
For each non-smooth point $z \in \R$ of the mapping $\beta:\R \to \R$, there exists $\delta_z>0$,  so that $\beta$
  is convex or  concave  on the interval $[z-\delta_z, z+\delta_z]$ and \ro{differentiable on $[z-\delta_z,z)\cup(z,z+\delta_z]$. In addition, the set of non-differentiable points $\NN$ is well-ordered and the intervals $[z-\delta_z, z+\delta_z],\ z \in \NN$, are disjoint}. Moreover, $\beta$ is continuously differentiable outside the set $\cup_{z\in \NN} [z-\delta_z/4, z+\delta_z/4]$.\end{assumption}
  \ro{\begin{remark}The proof of Lemma \ref{lem:tool} below shows that it   suffices if the set $\NN \cap [-\|\bar y\|_{\li},\|\bar y\|_{\li}]$ is well-ordered, where $\bar y $ is  the optimal state associated to a local optimum of \eqref{p10}. For instance, the situation where $\NN=\Z$ is allowed (though excluded in Assumption \ref{assu:reg}). \end{remark}}

In the rest of the paper, Assumption \ref{assu:reg} is tacitly assumed without mentioning it every time.
\begin{lemma}\label{lem:tool}
Suppose that Assumption \ref{assu:reg} is true. 
Let $\bar g \in \FF$ be locally optimal for \eqref{p10} in the sense of Definition \ref{def_e} and denote by $\bar y \in \hd \cap H^2(D)$ its associated state. Then, there exist an 
 adjoint state $p \in \hd \cap H^2(D)$ and a multiplier $\zeta \in \li$ so that 
 the optimality system is satisfied:
 \begin{subequations} \label{eq:lem}  \begin{gather}
-\laplace p+\zeta p+\frac{1}{\eps}H_\eps(\bar g)p=2\chi_E(\bar y-y_d) \ \text{ in }D, \quad p=0 \text{ on }\partial D, \label{lem:adj}
\\\zeta(x) \in [\min\{\beta_-'(\yy(x)),\beta_+'(\yy(x))\},\max\{\beta_-'(\yy(x)),\beta_+'(\yy(x))\}] \quad \text{a.e.\,in }D,\label{lem:clarke0}
\\( p[\eps-\frac{1}{\eps}H_\eps'(\bar g)\bar y]-\alpha H_\eps'( \bar g),h-\bar g)_{\ld}+(\bar g-\bar g_{\mathfrak{sh}},h-\bar g)_{\WW}\geq 0\quad \forall\,h\in \FF ,\label{lem:grad_f0}
   \end{gather}\end{subequations}
   where, for an arbitrary $z\in \R$, the left and right-sided derivative of $\beta: \R \to \R$  are defined through
$\beta'_-(z) := -\beta'(z;-1)$ and $ \beta'_+(z) :=\beta'(z;1)$, respectively. 
   \end{lemma}
\begin{proof}
See Appendix \ref{sec:a}.
   \end{proof}

\begin{remark}[Correspondence to KKT conditions]\label{rem:kkt}
If $\bar y(x) \not \in \NN$ f.a.a.\ $x \in D$, where $\NN$ denotes the set of non-smooth points of $\beta$, then $\zeta =\beta ' ( \bar y)$ f.a.a.\ $x \in D$. In this case, \eqref{eq:lem}
is equivalent to the optimality system or standard KKT-conditions, which one would obtain if one would assume $\beta$ to be continuously differentiable.
These conditions are given by :
\begin{subequations}\label{eq:kkt}
 \begin{gather}
-\laplace p+\beta'(\yy)  p+\frac{1}{\eps}H_\eps(\bar g)p=2\chi_E(\bar y-y_d)\  \text{ in }D, \quad p=0 \text{ on }\partial D, 
\\( p[\eps-\frac{1}{\eps}H_\eps'(\bar g)\bar y]-\alpha H_\eps'( \bar g),h-\bar g)_{\ld}+(\bar g-\bar g_{\mathfrak{sh}},h-\bar g)_{\WW}\geq 0\quad \forall\,h\in \FF.\label{grad_f}
 \end{gather}
 \end{subequations}
\end{remark}
\begin{remark}\label{rem:sign}
Despite of Remark \ref{rem:kkt},  the optimality conditions \eqref{eq:lem} are incomplete, if $\beta$ has non-differentiable points that are attained by $\bar y$ on a set of positive measure. For a comparison, we refer to \eqref{eq:thm} below, which is of strong stationary type under Assumption \ref{assu:cq}, cf.\,Theorem \ref{thm:equiv_B_strong}. 
We emphasize that the non-smooth case requires  a sign condition for the adjoint state in those points $x$ where the argument of  the non-smoothness $\beta$ in the state equation, in our case $ \bar y(x)$, is such that $\beta$ is not differentiable at $\bar y(x)$. The respective sign depends on  the convexity/concavity properties of $\beta$ around non-smooth points (Remark  \ref{rem:needed} below). This is what ultimately distinguishes a strong stationary optimality system from very `good'    optimality systems obtained by smoothening procedures such as \eqref{eq:lem}.
If the aforementioned sign condition is not available, 
 the equivalence of the optimality system to the first order necessary optimality condition, and thus, its strong stationarity, fails, see also the proof of Theorem \ref{thm:equiv_B_strong} and Remark \ref{rem:needed} below. 
This fact has been observed in many contributions dealing with strong stationarity \cite[Rem.\ 6.9]{paper}, \cite[Rem.\ 3.9]{mcrf}, \cite[Rem.\ 4.8]{st_coup}, \cite[Rem.\ 4.15]{cr_meyer}, \cite[Rem.\,3.20]{jnsao}, \cite[Rem.\,4.14]{amo}.
\end{remark}

\subsection{B-stationarity}\label{b}
In light of Remark \ref{rem:sign}, we continue our investigations by writing down the first order necessary optimality condition in primal form for a local optimum of \eqref{p10}.
First, we need to examine the differentiability properties of  the solution map associated to the state equation in \eqref{p10}. We recall that this reads 
\begin{equation}\label{eq}
  \begin{aligned}
   -\laplace y + \beta(y)+\frac{1}{\eps}H_\eps(g)y&=f +\eps g\quad \text{in }D,
   \\y&=0\quad \text{on } \partial D.
 \end{aligned}
\end{equation}
\begin{lemma}[Directional differentiability of $S$]\label{lem:dirderiv}
 The control-to-state operator $S: L^2(D) \to H^1_0(D) $  is directionally differentiable. Its directional derivative 
 $u := S'(g;h)$ at $g\in L^2(D)$ in direction $h\in L^2(D)$ belongs to $ H^1_0(D) \cap H^2(D)$ and is given by the unique  solution of 
\begin{equation}\label{lin1} 
  \begin{aligned}[t]
   -\laplace u + \beta'(y;u)+\frac{1}{\eps}H_\eps(g)u+\frac{1}{\eps}H_\eps'(g)yh&=\eps h\quad \text{in }D,
   \\u&=0 \quad \text{on } \partial D,\end{aligned} 
\end{equation}
 where  $y := S(g)$.
\bl{Moreover, for each $g\in \ld$, the  operator $S'(g;\cdot):\ld \to H^1_0(D)$ is globally Lipschitz continuous in the following sense: $\forall\,M>0$ it holds
  \begin{equation}\begin{aligned}\label{s'lip}
  &{ \|S'(g; h_1) - S'(g; h_2)\|_{H^1_0(D)}} \\&\qquad  \leq L_{M+1} \, \|h_1 - h_2\|_{L^2(D)} \quad \forall\,h_1,h_2 \in \ld,\ \forall\,g \in \clos{B_{\ld}(0,M)},
  \end{aligned}\end{equation}where $L_{M}>0$ is the constant from \eqref{slip}
.}
\end{lemma}
\begin{proof} 
 Let $\tau >0$ be arbitrary, but fixed, small enough, and define for simplicity $y^\tau:=S(g+\tau h)$ and $u^\tau:=\frac{y^\tau-y}{\tau}$. Then, by the  Lipschitz continuity of $S$, see \eqref{slip}, we have 
\begin{equation*}\label{eq:estim_eta}\|u^\tau\|_{ H^1_0(D) \cap H^2(D)} \leq L \|h\|_{\ld},\end{equation*} 
where $L>0$ depends on $g$ and $h$ only.
Thus, we can extract a weakly convergent  subsequence   $\{u^\tau\}_{\tau}$ with 
\begin{equation}\label{eq:conv_eta}
u^\tau \weakly \tilde u \quad \text{ in }H^1_0(D) \cap H^2(D) \quad \text{as }\tau \searrow 0.
\end{equation} 
By subtracting the equation \eqref{eq} associated to $g$ from \eqref{eq} associated to  $g+\tau h$  we obtain  
\begin{equation}\label{eq:awa_diff2}
   -\laplace u^\tau + \frac{\beta(y^\tau)-\beta(y) }{\tau}+\frac{1}{\eps}H_\eps(g+\tau h)u^\tau+\frac{1}{\eps}\frac{(H_\eps(g+\tau h)-H_\eps(g))y }{\tau}= \eps h.
   \end{equation}Further, by \eqref{est_y}, we have $\|y^\tau\|_{\li} \leq c_1+c_2 (\|g\|_{\lo}+\|h\|_{\lo}):=C \neq C(\tau)$. Then, with $M:=\max\{C, \|y\|_{\li}+\|\tilde u\|_{\li}\}$ in \eqref{eq:flip} we get
   \begin{equation}\label{eq:f_diff}\begin{aligned}
   \Big \|\frac{\beta(y^\tau)-\beta(y)}{\tau}-\beta'(y;\tilde u) \Big \|_{\ld}
 & \leq  \Big \|\frac{\beta(y^\tau)-\beta(y+\tau \tilde u)}{\tau}  \Big \|_{\ld}\\&\quad + \Big \|\frac{\beta(y+\tau \tilde u)-\beta(y) }{\tau} -\beta'(y; \tilde u)  \Big \|_{\ld}
  \\ &\leq  L_M\,  \|u^\tau -  \tilde u \|_{\ld}
  \\&\quad + \Big \|\frac{\beta(y+\tau \tilde u)-\beta(y) }{\tau} -\beta'(y; \tilde u)  \Big \|_{\ld}
   \\&\quad \longrightarrow   0 \quad \text{as }\tau \searrow 0.
\end{aligned}\end{equation} Note that the convergence in \eqref{eq:f_diff} is due to \eqref{eq:conv_eta}, and the directional differentiability of $\beta: \li \to \lo$. The latter follows by Assumption \ref{assu:stand}.\ref{it:stand3}, Lebesgue dominated convergence theorem and the local Lipschitz continuity of $\beta$ (Assumption \ref{assu:stand}.\ref{it:stand2}).
Now, passing to the limit in \eqref{eq:awa_diff2}, where we rely on \eqref{eq:conv_eta},  \eqref{eq:f_diff} and the continuity and \bl{G\^ateaux-}differentiability of $H_\e:\lo \to \lo$ (Remark \ref{rem:h_e}), yields that $\tilde u$ solves \eqref{lin1}.  As \eqref{lin1} admits at most one solution (in view of the monotonicity of $\beta$ and since $H_\eps(g) \geq 0$), we can now deduce the desired directional differentiability result. \bl{Finally, to show the last assertion, we use \eqref{slip}, from which we infer
 \begin{equation}
  \frac{ \|S(g+\tau h_1) - S(g+\tau h_2)\|_{H^1_0(D)}}{\tau} \leq L_{M+1} \, \|h_1 - h_2\|_{L^2(D)} \quad \forall\, g \in \clos{B_{\ld}(0,M)},
  \end{equation}for all $h_1,h_2 \in \clos{B_{\ld}(0,1)}$ and $\tau \in (0,1].$ Using the directional differentiability of $S:\ld \to H_0^1(D),$ we then arrive at
   \begin{equation}\begin{aligned}
 & { \|S'(g; h_1) - S'(g; h_2)\|_{H^1_0(D)}}\\&\qquad \leq L_{M+1} \, \|h_1 - h_2\|_{L^2(D)} \quad \forall\, g \in \clos{B_{\ld}(0,M)}, \ \forall\,h_1,h_2 \in \clos{B_{\ld}(0,1)}.\end{aligned}
  \end{equation} For arbitrary $h_1,h_2 \in \ld$, define $m:=\max\{\|h_1\|,\|h_2\|\},$ apply the above inequality to
$h_1/m$ and $h_2/m$, and use the positive homogeneity of $S'(g;\cdot)$.}
The proof is now complete.   \end{proof}

With the directional differentiability of $S$ at hand, we may now state the first order necessary optimality condition for a local optimum of \eqref{p10}. In what follows, this is often referred to as B-stationarity, this terminology having its roots in  \cite{ScheelScholtes2000}.

 \begin{lemma}[B-stationarity] If $\bar g \in \FF$ is locally optimal for \eqref{p10} in the sense of Definition \ref{def_e} with associated state $\yy$, then it holds
 \begin{equation}\label{eq:vi}
(2\chi_E(\bar y-y_d),S'(\bar g; h-\bar g))_{L^2(D)}-(\alpha H_\eps'(\bar g),h-\bar g)_{\ld} +(\bar g-\bar g_{\mathfrak{sh}},h-\bar g)_{\WW}\geq 0
\end{equation}for all $h\in \FF.$   \end{lemma}
\begin{proof}
In view of Lemma \ref{lem:dirderiv} and chain rule (see e.g.\,\cite{shapiro}), the reduced objective $j:\ld \to \R$ (cf.\,Definition \ref{def_e}) is directionally differentiable. As the local optimality of $\bar g$ implies $j'(\bar g; h-\bar g)\geq 0$ for all $h\in \FF$, we arrive at \eqref{eq:vi}.
\end{proof}

  \subsection{Strong stationarity}\label{sec:ss}
  This subsection is dedicated to the improvement of the optimality system \eqref{eq:lem}. Starting from the B-stationarity established in the previous subsection, we aim at proving all the optimality conditions that are missing in  \eqref{eq:lem}. These consist of sign conditions for the adjoint state in those points $x$ for which the state is a non-differentiable point of $\beta$ (Remark \ref{rem:needed}). As shown by Theorem \ref{thm:equiv_B_strong} below, the optimality system obtained in this subsection, i.e., \eqref{eq:thm}, does not lack any information, provided that certain CQs are satisfied (Assumptions \ref{assu:h_a} and  \ref{assu:cq}). 
  
  
  
In all what follows,  $\bar g \in \FF$ is fixed, locally optimal for \eqref{p10} in the sense of Definition \ref{def_e}, while $\bar y \in \hd \cap H^2(D)$ is its associated state. Assumption \ref{assu:reg} is tacitly assumed the whole time.

We recall that the  non-smooth control problem reads as follows:
\begin{equation}\tag{$P_{\eps}$}
 \left.
 \begin{aligned}
  \min_{g \in  \FF } \quad & \int_E (y(x)-y_d(x))^2 \; dx  +\alpha\,\int_{D} (1-H_\eps(g))(x)\,dx+\frac{1}{2}\,\|g-\bar g_{\mathfrak{sh}}\|_{\WW}^2
  \\   \text{s.t.} \quad & 
  \begin{aligned}[t]
   -\laplace y + \beta(y)+\frac{1}{\eps}H_\eps(g)y&=f +\eps g\quad \text{ in }D,
   \\y&=0\quad \text{ on } \partial D. \end{aligned} 
 \end{aligned}
 \quad \right\}
\end{equation}

   \begin{definition}In all what follows we denote by $\NN$ the set of non-smooth points of $\beta$ and  $$\dn:=\{x \in D: \bar y (x) \in \NN, \beta \text{ is convex around }\bar y(x)\}.$$
   $$\dnn:=\{x \in D: \bar y (x) \in \NN, \beta \text{ is concave around }\bar y(x)\}.$$   \end{definition}

   \begin{lemma}\label{sign_b}
Under  Assumption \ref{assu:reg},  it holds
    $$\beta'_+(\yy(x))>\beta'_-(\yy(x)) \geq 0 \quad \forall x \in \dn,$$
        $$\beta'_-(\yy(x))>\beta'_+(\yy(x)) \geq 0 \quad \forall x \in \dnn,$$
         where, for an arbitrary $z\in \R$, the left and right-sided derivative of $\beta: \R \to \R$  are defined through
$\beta'_-(z) := -\beta'(z;-1)$ and $ \beta'_+(z) :=\beta'(z;1)$, respectively. 
   \end{lemma}
   \begin{proof}
   We only show the first assertion, as the second is proven completely analogously.
   The non-negativity is due to the fact that $\beta$ is \ro{non-decreasing} (Assumption \ref{assu:stand}). Further, by the definition of convexity, we have
   \[
{\frac{\beta(\yy(x)+ \tau)+\beta(\yy(x)-\tau)}{ 2}  \geq \beta(\yy(x))}  \quad \forall\, x \in \dn
\]for $\tau>0$ small enough (see Assumption \ref{assu:reg}). We rearrange the terms and arrive at
$$\frac{\beta(\yy(x)+ \tau)-\beta(\yy(x))}{ \tau} \geq 
\frac{\beta(\yy(x))-\beta(\yy(x)-\tau)}{ \tau} \quad \forall\, x \in \dn.
$$
Passing to the limit $\tau \searrow 0$ then yields 
   $$\beta'_+(\yy(x))\geq \beta'_-(\yy(x)) \quad \forall x \in \dn.$$Since $\yy (x) \in \NN,$ the above inequality is strict. 
This completes the proof.
   \end{proof}  

   \begin{definition}The set where the control constraint is active, cf.\,the definition of $ \FF$, is denoted by
   $$\AA:=\{x \in E:\bar g(x)=0\}.$$
  Note that this is well-defined up to a set of measure zero. Whenever we refer to $\clos \AA$ in what follows, we mean the closure of the set $\AA$ associated to a representative of $\bar g$.
     \end{definition} 

     Before we proceed with the proof of the main result of this section, we state the first "constraint qualification", which concerns the non-linear term acting on the control in the state equation.
     \begin{assumption}["Constraint Qualification"]\label{assu:h_a}
We assume that the set $\{x \in D \setminus \bar E: {\frac{1}{\eps}H_\eps'(\bar g)\bar y-\eps}=0\}$ has measure zero and 
$$\int_{D \setminus \bar E} \frac{1}{|\frac{1}{\eps}H_\eps'(\bar g)\bar y-\eps|^2} \,dx <\infty.$$
\end{assumption}

\begin{remark}\label{rem:cq0}
The necessity of Assumption \ref{assu:h_a} is owed to the fact that the control  enters the state equation \eqref{eq} in a non-linear fashion via the mapping $H_\e$. This specific CQ  basically ensures the surjectivity of the operator ${\frac{1}{\eps}H_\eps'(\bar g)\bar y-\eps}:\ld \to L^\infty( D)$, see \eqref{h}, \eqref{h_m} and part (II) in the proof of Theorem \ref{thm} below.
\end{remark}

\begin{theorem}\label{thm}
Suppose that Assumptions \ref{assu:reg} and \ref{assu:h_a}  are true. 
Let $\bar g \in \FF$ be  a local optimum of \eqref{p10} in the sense of Definition \ref{def_e} and denote by $\bar y \in \hd \cap H^2(D)$ its associated state.  Then, there exists an 
 adjoint state $p \in \hd \cap H^2(D)$ and a multiplier $\zeta \in \li$ so that 
 the following optimality system is satisfied \ro{for each  representative associated to $\bar g$}:
 \begin{subequations} \label{eq:thm}  \begin{gather}
-\laplace p+\zeta p+\frac{1}{\eps}H_\eps(\bar g)p=2\chi_E(\bar y-y_d) \text{  in }D, \quad p=0 \text{ on }\partial D, \label{adj}
\\\zeta(x) \in [\min\{\beta_-'(\yy(x)),\beta_+'(\yy(x))\},\max\{\beta_-'(\yy(x)),\beta_+'(\yy(x))\}] \text{ a.e.\,in } D,\label{clarke0}
\\p \leq 0 \ \text{ a.e.\,in } D_n^{convex}\ro{\setminus \clos \AA},\label{sign_p0}
\\p \geq 0 \ \text{ a.e.\,in } D_n^{concave}  \setminus \clos \AA,\label{sign_p1}
\\( p[\eps-\frac{1}{\eps}H_\eps'(\bar g)\bar y]-\alpha H_\eps'( \bar g),h-\bar g)_{\ld}+(\bar g-\bar g_{\mathfrak{sh}},h-\bar g)_{\WW}\geq 0\quad \forall\,h\in \FF. \label{grad_f0}
   \end{gather}\end{subequations}Moreover, 
   \begin{equation}\label{sign_p}
   p \leq 0\quad  \ro{\ae } \AA.\end{equation}
    If the set $\AA$ has measure zero, then \eqref{sign_p0}--\eqref{sign_p1} is replaced by   \ro{ \begin{equation}\label{sign_p'}
  p \leq 0 \ \text{ a.e.\,in } D_n^{convex},\quad    p \geq 0 \ \text{ a.e.\,in } D_n^{concave}.\end{equation}}
\end{theorem}
\begin{remark}\label{rem:s}
Compared to classical optimality systems obtained via smoothening procedures \cite{tiba}, the  conditions in \eqref{eq:thm} are stronger, as they contain information about the sign of the adjoint state in points $x \in D$ where the non-smoothness is active (cf.\,Remark \ref{rem:sign}). 

If $\beta$ is locally convex at each $z \in \NN$ for which $\{\bar y=z\}$ has positive measure, then \eqref{eq:thm} is of strong stationary type \ro{for a representative of $\bar g \in \FF$}, \ro{provided that $\mu(\dn \cap (\clos \AA \setminus \AA))=0$ where $\AA$ is the set associated to the respective representative} (see Theorem \ref{thm:equiv_B_strong} below). If $\beta$ happens to be concave at some non-smooth points that are attained by $\bar y$, then strong stationarity is guaranteed under the additional assumption that the set $D_n^{concave} \cap  \clos \AA$ has measure zero (or if $\AA$ has measure zero), cf.\,Assumption \ref{assu:cq} and Remark \ref{rem:needed} below.\end{remark} 
   \begin{proof}
   In view of Lemma \ref{lem:tool} we only need to
 show  the sign conditions \eqref{sign_p0}-\eqref{sign_p1}. We begin by taking a closer look at the information gathered from \eqref{grad_f0}. 
 Testing with $h:=\bar g+\chi_{D\setminus E} \phi,$ where $\phi \in H^s(D\setminus \bar E ),$ yields
  \begin{equation}\label{eq_d-e}
  ( p[\eps-\frac{1}{\eps}H_\eps'(\bar g)\bar y]-\alpha H_\eps'( \bar g),\phi)_{L^2(D\setminus E)}+(\bar g-\bar g_{\mathfrak{sh}},\phi)_{L^2(D\setminus E)}+(\bar g-\bar g_{\mathfrak{sh}},\phi)_{\hde}= 0.  \end{equation}
Now let $v \in L^2(E),\ v \leq 0 \text{ a.e.\ in }E$, be arbitrary but fixed. By testing  \eqref{grad_f0} with $h:=\chi_{E}v+\chi_{D\setminus E} \bar g$, one obtains 
 \begin{equation}
 \int_E (p[\eps-\frac{1}{\eps}H_\eps'(\bar g)\bar y]-\alpha H_\eps'( \bar g)+\bar g-\bar g_{\mathfrak{sh}})(v-\bar g) \,dx \geq 0.
\end{equation}This results in
\begin{equation}\label{cone}
 \int_E (p[\eps-\frac{1}{\eps}H_\eps'(\bar g)\bar y]-\alpha H_\eps'( \bar g)+\bar g-\bar g_{\mathfrak{sh}})\o \,dx \geq 0 \quad \forall\,\o\in \overline{\cone(\widetilde \FF  -\bar g)}^{L^2(E)},
\end{equation}where \[\widetilde \FF:=\{v \in L^2(E):v \leq 0 \text{ a.e.\,in }E\}\]
and 
$\cone (\widetilde \FF  -\bar g)$ is the conical hull of the set $\widetilde \FF  -\bar g$. By  \cite[Lem.\,6.34]{bs}, it holds
\begin{equation}\label{def_c}\begin{aligned}
\overline{\cone(\widetilde \FF  -\bar g)}^{L^2(E)}=\{\o \in L^2(E):  \o \leq 0 \text{ a.e.\ in }\AA\}.
\end{aligned}
\end{equation}
Next,  let $ \phi \in C_c^\infty(E)$ be arbitrary, but fixed. We test 
 \eqref{cone}  with $\chi_{E \setminus \AA}\phi $. Then, the fundamental lemma of calculus of variations gives in turn  
 \begin{equation}\label{equality}
 p[\eps-\frac{1}{\eps}H_\eps'(\bar g)\bar y]-\alpha H_\eps'( \bar g)+\bar g-\bar g_{\mathfrak{sh}} = 0\quad  \text{ a.e.\,in } E \setminus \AA.
 \end{equation}
Similarly, let $ v \in C_c^\infty(E), \ v \leq 0$ be arbitrary, and we test 
 \eqref{cone}  with $\chi_{\AA}v,$ whence $$p[\eps-\frac{1}{\eps}H_\eps'(\bar g)\bar y]-\alpha H_\eps'( \bar g)+\bar g-\bar g_{\mathfrak{sh}} \leq 0\quad  \text{ a.e.\,in } \AA$$ follows.
By \eqref{reg_h'}, the definition of $\AA$, and since $\bar g_{\mathfrak{sh}} \in \FF$,  this implies
 \begin{equation}\label{s_p}
  \eps p\leq \bar g_{\mathfrak{sh}}\leq 0 \quad \text{a.e.\,in  }\AA. \end{equation}
Next we distinguish between the cases when $\AA$ has positive measure (I) and when $\AA$ has measure zero (II).
  \\(I) 
Now let $ h \in \FF$ be arbitrary but fixed and  test \eqref{adj} with $u:=S'(\bar g;h-\bar g)$ and \eqref{lin1} with right hand side $h-\bar g$ with $p$. Then, in view of \eqref{eq:vi}, we have   \begin{equation}\begin{aligned}\label{necc}
0&\leq (2\chi_E(\bar y-y_d),u)_{\ld}-(\alpha H_\eps'(\bar g),h-\bar g)_{\ld} +(\bar g -\bar g_{\mathfrak{sh}},h-\bar g)_{\WW} 
\\&\overset{\eqref{adj}}{=}(-\laplace p+\zeta p +\frac{1}{\eps}H_\eps(\bar g)p,u)_{\ld}-(\alpha H_\eps'(\bar g),h-\bar g)_{\ld} +(\bar g -\bar g_{\mathfrak{sh}},h-\bar g)_{\WW} 
\\&=-(\laplace u,p)_{L^2(D)}+(\zeta p ,u)_{L^2(D)}+\frac{1}{\eps}(H_\eps(\bar g)p,u)_{L^2(D)}
\\&\quad -(\alpha H_\eps'(\bar g),h-\bar g)_{\ld} +(\bar g -\bar g_{\mathfrak{sh}},h-\bar g)_{\WW} 
\\&\overset{\eqref{lin1}}{=}(\zeta p ,u)_{L^2(D)}-(\beta'(\bar y; u),p)_{L^2(D)}
-\frac{1}{\eps}(H_\eps'(\bar g)\bar yp,h-\bar g)_{L^2(D)}+\eps(h-\bar g,p)_{L^2(D)}
\\&\quad -(\alpha H_\eps'(\bar g),h-\bar g)_{\ld} +(\bar g -\bar g_{\mathfrak{sh}},h-\bar g)_{\WW}
\\&=(\zeta p ,u)_{L^2(D)}-(\beta'(\bar y; u),p)_{L^2(D)}
\\&\quad+( p[\eps-\frac{1}{\eps}H_\eps'(\bar g)\bar y]-\alpha H_\eps'( \bar g)+\bar g-\bar g_{\mathfrak{sh}},h-\bar g)_{L^2(D\setminus E)}+(\bar g-\bar g_{\mathfrak{sh}},h-\bar g)_{\hde}
\\&\quad+( p[\eps-\frac{1}{\eps}H_\eps'(\bar g)\bar y]-\alpha H_\eps'( \bar g)+\bar g-\bar g_{\mathfrak{sh}},h-\bar g)_{L^2(E\setminus \AA)}
\\&\quad+( p[\eps-\frac{1}{\eps}H_\eps'(\bar g)\bar y]-\alpha H_\eps'( \bar g)+\bar g-\bar g_{\mathfrak{sh}},h-\bar g)_{L^2( \AA)}.
\end{aligned} \end{equation}
In light of \eqref{eq_d-e} and \eqref{equality}, this leads to 
\begin{equation}\begin{aligned}\label{necc0}
0\leq
&(\zeta p ,S'(\bar g;h-\bar g))_{L^2(D)}-(\beta'(\bar y; S'(\bar g;h-\bar g)),p)_{L^2(D)}
\\&\quad+( p[\eps-\frac{1}{\eps}H_\eps'(\bar g)\bar y]-\alpha H_\eps'( \bar g)+\bar g-\bar g_{\mathfrak{sh}},h-\bar g)_{L^2( \AA)} \quad \forall\,h \in \FF.
\end{aligned} \end{equation}
Testing again with $h:=\bar g+\chi_{D\setminus E} \phi,$ where $\phi \in H^s(D\setminus \bar E ),$ and taking the dense embedding  $\hde \dense L^2(D\setminus E)$ together with the continuity of $S'(\bar g;\cdot):\ld\to \ld$ into account (\bl{cf.\,\eqref{s'lip} in Lemma \ref{lem:dirderiv}}),  further implies
\begin{equation}\begin{aligned}\label{necc001}
0\leq
&(\zeta p ,S'(\bar g; \phi))_{L^2(D)}-(\beta'(\bar y; S'(\bar g;\phi)),p)_{L^2(D)} \quad \forall\,\phi \in L^2(D), \phi=0 \text{ a.e.\,in }E.
\end{aligned} \end{equation}
Now, let $\psi \in C^\infty_c(D), \psi=0$ in $\bar E,$ be arbitrary, but fixed, and define 
\begin{equation}\label{h}
\widehat h:=\frac{-\laplace \psi + \beta'(\bar y;\psi)+\frac{1}{\eps}H_\eps(\bar g)\psi}{\eps-\frac{1}{\eps}H_\eps'(\bar g)\bar y}. \end{equation}
Thanks to Assumption \ref{assu:h_a}, the above definition makes sense and $\widehat h \in \ld,$ as the function in the numerator belongs to $\li.$
 Since \eqref{lin1} is uniquely solvable, this means that
$$S'(\bar g;\widehat h)=\psi.$$ Note that $\psi=0$ on $\partial D$. Since $\widehat h=0$ in $E$, we can test with $\widehat h$ in \eqref{necc001}. This implies 
 \begin{equation}\begin{aligned}
 0\leq (\zeta p ,\psi)_{L^2(D\setminus E)}-(\beta'(\bar y; \psi),p)_{L^2(D\setminus E)} \quad \forall \psi \in C^\infty_c(D \setminus \bar E).
\end{aligned} \end{equation}By applying  the fundamental lemma of calculus of variations and by making use of the positive homogeneity of the directional derivative w.r.t.\,direction, we obtain
\begin{equation}\begin{aligned}\label{f_flcv}
\zeta p &\geq \beta'_+(\bar y)p \quad \text{a.e.\,in }D \setminus \bar E,
\\
\zeta p &\leq \beta'_-(\bar y)p \quad \text{a.e.\,in }D \setminus \bar E.
\end{aligned}
\end{equation}
Now we go back to \eqref{necc0} and let $v \in L^2(\bl{D}),\ v \leq 0 \text{ a.e.\,in }E,$ be arbitrary but fixed. By testing  \eqref{necc0} with $h:=\chi_{E}v+\chi_{D\setminus E} \bar g$ one arrives at
 \begin{equation}\begin{aligned}
0\leq
&(\zeta p ,S'(\bar g;\chi_{E}(v-\bar g)))_{L^2(D)}-(\beta'(\bar y; S'(\bar g;\chi_{E}(v-\bar g))),p)_{L^2(D)}
\\&\quad+( p[\eps-\frac{1}{\eps}H_\eps'(\bar g)\bar y]-\alpha H_\eps'( \bar g)+\bar g-\bar g_{\mathfrak{sh}},v-\bar g)_{L^2( \AA)}.
\end{aligned} \end{equation}
Using again density arguments, the positive homogeneity and the continuity of the directional derivative w.r.t.\,direction, then  yields
  \begin{equation}\begin{aligned}\label{ne00}
0\leq
&(\zeta p ,S'(\bar g;\bl{\mathcal{E} (h)} ))_{L^2(D)}-(\beta'(\bar y; S'(\bar g;\bl{\mathcal{E} (h)} )),p)_{L^2(D)}
\\&\quad+( p[\eps-\frac{1}{\eps}H_\eps'(\bar g)\bar y]-\alpha H_\eps'( \bar g)+\bar g-\bar g_{\mathfrak{sh}},h)_{L^2( \AA)} \quad \forall\,h\in \overline{\cone(\widetilde \FF  -\bar g)}^{L^2(E)},
\end{aligned} \end{equation}
 where we recall \begin{equation*}
\overline{\cone(\widetilde \FF  -\bar g)}^{L^2(E)}=\{\o \in L^2(E):  \o \leq 0 \text{ a.e.\ in }\AA\},
\end{equation*}see \eqref{def_c}. 
\bl{Here, $\bl{\mathcal{E} (h)} \in \ld$ denotes the extension by zero on $D$ of $h\in L^2(E)$.}
We proceed with similar arguments to those used in  the proof of \eqref{f_flcv}. Let $m\in \N_+$. Let $\AA$ be the (uniquely specified) set associated to an arbitrary, but fixed, representative of $\bar g.$ 
\bl{By Assumption \ref{assu:E}, the set $\clos D\setminus E$ is closed. From \cite[Lem.\,A.1]{cr_meyer} we then deduce that there exists }
$\psi_m \in C^\infty(D), \psi_m=0$ in $(\clos D\setminus E) \cup \clos \AA_m,$  $\psi_m>0$ in $E \setminus \clos \AA_m,$
where 
\begin{equation}\label{am}
\AA_m:=\{x \in D:d(\AA,x)<1/m\}.\end{equation}
We notice that $\AA_m$ is an open set. 
Define 
\begin{equation}\label{h_m}
\widehat h_m:=\frac{-\laplace( \phi \psi_m )+ \beta'(\bar y;\phi \psi_m)+\frac{1}{\eps}H_\eps(\bar g)\phi \psi_m}{\eps-\frac{1}{\eps}H_\eps'(\bar g)\bar y}, \end{equation} where $\phi \in C_c^\infty(D)$ is arbitrary, but fixed.
Thanks to Assumption \ref{assu:h_a}, the above definition makes sense and $\widehat h_m \in \ld,$ as the function in the numerator belongs to $\li.$
 Since \eqref{lin1} is uniquely solvable, this means that
\begin{equation}\label{idd}
S'(\bar g;\widehat h_m)=\phi \psi_m=S'(\bar g;\chi_E \widehat h_m).\end{equation}
Note that $\psi_m=0$ on $\partial D$ and $\widehat h_m=0$ \bl{a.e.\,in $D \setminus  E$} (recall that $d(\bar E,\partial D)>0$ and $\mu(\partial E)=0$, by Assumption \ref{assu:E}). Since $\widehat h_m=0$ in $\AA_m \supset \clos \AA,$ cf.\,\eqref{am}, we can test with $ \chi_E \widehat h_m$ in \eqref{ne00}. This gives in turn  
 \begin{equation}\begin{aligned}
0\leq
(\zeta p ,\phi  \psi_m)_{L^2(D)}-(\beta'(\bar y; \phi  \psi_m),p)_{L^2(D)}
 \quad \forall\,m \in \N_+.
\end{aligned} \end{equation}
As $ \psi_m=0$ in $(\clos D\setminus E) \cup \clos \AA_m,$  by definition, the above inequality is equivalent to 
 \begin{equation}\begin{aligned}
0\leq
(\zeta p ,\phi \psi_m)_{L^2(E \setminus \clos \AA_m)}-(\beta'(\bar y; \phi \psi_m),p)_{L^2(E \setminus \clos \AA_m)}
 \quad \forall\,m \in \N_+.
\end{aligned} \end{equation}
By applying again the fundamental lemma of calculus of variations and by making use of the positive homogeneity of the directional derivative w.r.t.\,direction, we obtain
$$\zeta p \geq \beta'_+(\bar y)p \quad \text{a.e.\,in }E \setminus \clos \AA_m,$$
$$\zeta p \leq \beta'_-(\bar y)p \quad \text{a.e.\,in }E \setminus \clos \AA_m,$$ where we also employed the fact that $\psi_m>0$ in $E \setminus \clos \AA_m$, by definition.
Passing to the limit $m \to \infty$ then yields
\begin{equation}\begin{aligned}\label{f_flcv0}
\zeta p \geq \beta'_+(\bar y)p \quad \text{a.e.\,in }E \setminus \clos \AA,
\\
\zeta p \leq \beta'_-(\bar y)p \quad \text{a.e.\,in }E \setminus \clos \AA.
\end{aligned}
\end{equation}
Here we relied on the convergence \begin{equation}\label{chii}
\chi_{\clos \AA_m}(x) \to \chi_{\clos \AA}(x) \quad \forall \,x\in  D, \  \text{ as }m \to \infty, \end{equation}
which can be proven by arguing as in the proof of \cite[Lemma 3.6]{mcrf}. For convenience of the reader, we briefly recall the  arguments.
 From \eqref{am} one  infers that 
$$\raisebox{3pt}{$\chi$}_{\overline {\AA_{m}}}\,(x) = \raisebox{3pt}{$\chi$}_{\bar{\AA}}\,(x)=1 \quad  \ff  m \in \N_+, \quad  \text{if } x \in   \overline \AA.$$
If $x \not \in \overline \AA$, then $ d(\AA,x)>0$. Hence, there exists $m=m(x) \in \N_+$ with  $1/m(x) < d(\AA,x)$, i.e., 
$x \not \in  \overline {\AA_{m(x)}}.$ 
By  \eqref{am}, we see that $x \not \in  \overline {\AA_{m(x)+k}},$ for each $k\in \N,$ whence 
$$\raisebox{3pt}{$\chi$}_{\overline {\AA_{m(x)+k}}}\,(x) = 0 \quad  \ff  k \in \N \quad \text{and}\quad \raisebox{3pt}{$\chi$}_{\bar{\AA}}\,(x) = 0,  \quad   \text{if } x \not  \in   \overline \AA$$follows. Altogether, this implies \eqref{chii}. Now we go back to the proof, and we see that from \eqref{f_flcv} and \eqref{f_flcv0} in combination with Lemma \ref{sign_b} we can finally deduce 
\begin{subequations} \label{eq:1}  \begin{gather}
p \leq 0 \text{ a.e.\,in } D_n^{convex} \setminus \clos \AA,
\\p \geq 0 \text{ a.e.\,in } D_n^{concave} \setminus \clos \AA.
   \end{gather}\end{subequations}
 \ro{Together with \eqref{s_p},} this proves the sign conditions \eqref{sign_p0}-\eqref{sign_p1} and \eqref{sign_p}.

(II) If $\AA$ has measure zero, the proof simplifies considerably. One argues as in step (I) to obtain \eqref{f_flcv}. Since $\mu(\AA)=0,$ \eqref{ne00} now reads
\begin{equation}\begin{aligned}
0\leq
(\zeta p ,S'(\bar g;\chi_{E}h))_{L^2(D)}-(\beta'(\bar y; S'(\bar g;\chi_{E}h)),p)_{L^2(D)}
 \quad \forall\,h\in {L^2(E)}.
\end{aligned} \end{equation}Taking $\psi \in C_c^\infty(D),\ \psi=0$ in $\bar D \setminus E$, instead of $\psi_m$, and arguing in the exact same manner ultimately leads to 
$$\zeta p \geq \beta'_+(\bar y)p \quad \text{a.e.\,in }E,$$
$$\zeta p \leq \beta'_-(\bar y)p \quad \text{a.e.\,in }E.$$
In view of  Lemma \ref{sign_b}, this proves
\begin{subequations}   \begin{gather}
p \leq 0 \text{ a.e.\,in } D_n^{convex},
\\p \geq 0 \text{ a.e.\,in } D_n^{concave}.
   \end{gather}\end{subequations}
\end{proof}

\begin{remark}[Relevance of the $+\e g$ term in the state equation]\label{rem:eps}
A close inspection of the proof of Theorem \ref{thm} shows that the presence of the term $+\e g$ on the right hand side of  the state equation is fundamental for showing the sign conditions \eqref{sign_p0}-\eqref{sign_p1}. To see this, we take a look at the definitions of $\widehat h$ and $\widehat h_m,$ cf.\,\eqref{h} and \eqref{h_m}, respectively. If $\e g$ would not appear in \eqref{eq}, then the denominators of $\widehat h$ and $\widehat h_m$ would vanish whenever $\bar g \leq 0$, see  \eqref{reg_h'}. Besides, it is essential that we take $+\e g$, not $-\e g$ on the right hand side of  \eqref{eq}, otherwise the adjoint state $p$ has the "wrong" sign in $\dn,$ see also Remark \ref{rem:needed}. All the other results associated to this paper  including the approximation scheme in \cite{p1} stay true if the term $\e g$  is not considered at all  in \eqref{eq}. 
\end{remark}

The optimality system in Theorem \ref{thm} is of strong stationary type \ro{(for a representative, cf.\,Theorem \ref{thm:equiv_B_strong} below)}, provided that the following requirement is true.
\begin{assumption}["Constraint Qualification" because of control constraints]\label{assu:cq}\ro{Each local optimum  $\bar g \in \FF$ of \eqref{p10} has a representative so that }
\begin{equation}\label{cqm}
\ro{\mu[ (\dn \cap ( \clos \AA \setminus \AA))\cup(\dnn \cap \clos \AA)]=0},\end{equation}
where $\AA$ is  the set associated to this representative. \end{assumption}


\begin{theorem}[Strong stationarity]\label{thm:equiv_B_strong} 
Let $\bar g \in \FF$ together with its state $\bar y \in \hd \cap H^2(D) $, some 
 adjoint state $p \in \hd \cap H^2(D) $ and a multiplier $\zeta \in \li$ 
 satisfy the optimality system \eqref{eq:thm}-\eqref{sign_p} \ro{for a representative of $\bar g$. If this representative satisfies \eqref{cqm}, then it fulfills the variational inequality \eqref{eq:vi}}.
  If Assumptions \ref{assu:reg} and \ref{assu:h_a} are also true, then \eqref{eq:thm}-\eqref{sign_p} is equivalent to \eqref{eq:vi} \ro{for this particular representative}. 
 \end{theorem} 

\begin{proof}
As a result of \eqref{clarke0}-\eqref{sign_p1}, \eqref{sign_p}, in combination with Lemma \ref{sign_b} and \eqref{cqm}, one has 
\begin{equation}\begin{aligned}\label{des}
 \zeta(x) p(x) \in [\beta_+'(\yy(x))p(x),\beta_-'(\yy(x))p(x)] \quad \ae  D.
\end{aligned}
\end{equation}
Note that $\zeta(x)=\beta_-'(\yy(x))=\beta_+'(\yy(x))$ a.e.\,in $D \setminus (\dn \cup \dnn)$, by \eqref{clarke0}. Now, let $h \in \FF$ be arbitrary but fixed and abbreviate
$\eta:=S'(\bar g;h-\bar g)$ as well as $\eta^+:=\max\{\eta,0\}$ and $\eta^-:=\min\{\eta,0\}$. 
From \eqref{des} we deduce
$$  \zeta p \eta^+ \geq  \beta_+'(\yy)p\eta^+, \quad  \zeta p \eta^- \geq  \beta_-'(\yy)p\eta^- \quad \ae {D},$$
  which results in 
$$\int_D  \zeta p (\eta^+ + \eta^-) \,dx   \geq 
\int_{\{x \in D: \eta(x) > 0\}}  \beta_+'(\yy)p\eta^+ \,dx + \int_{\{x \in D: \eta(x) <0\}}  \beta_-'(\yy)p \eta^- \,dx.
$$
  Thus,   since $\eta=\eta^+ + \eta^-$, we have 
\begin{equation}\label{in0}
(\zeta p, \eta)_{\ld}-(\beta'(\yy;\eta), p)_{\ld} \geq 0 \quad \forall\,h \in \FF.
\end{equation}
In light of   \eqref{adj} and \eqref{lin1} with right hand side $h-\bar g$ tested with $\eta$ and $p$, respectively,  \eqref{in0} implies 
\begin{equation}\begin{aligned}
&0 \leq (2\chi_E(\bar y-y_d),\eta)_{L^2(D)}-(\frac{1}{\eps}H_\eps(\bar g)p,\eta)_{L^2(D)}+(\laplace p,\eta)_{L^2(D)}
\\&\quad -(\laplace \eta,p)_{L^2(D)}+(\frac{1}{\eps}H_\eps(\bar g)\eta+\frac{1}{\eps}H_\eps'(\bar g)\bar y (h-\bar g)-\eps\, (h-\bar g),p)_{\ld}
\\&=(2\chi_E(\bar y-y_d),\eta)_{L^2(D)}+(\frac{1}{\eps}H_\eps'(\bar g)\bar y p-\eps p,h-\bar g)_{\ld}
\\&\overset{\eqref{grad_f0}}{\leq }(2\chi_E(\bar y-y_d),\eta)_{L^2(D)}-\alpha(H_\eps'(\bar g),h-\bar g)_{\ld}+(\bar g-\bar g_{\mathfrak{sh}},h-\bar g)_{\WW} \quad \forall\,h \in \FF.
\end{aligned} \end{equation}
This completes the proof of the first statement. The converse assertion is due to Theorem \ref{thm}. Here we observe that the only information used about the local minimizer $\bar g \in \FF$ in the proof of Theorem \ref{thm} is contained in \eqref{eq:vi}.
\end{proof}

\begin{remark}\label{rem:needed}
The proof of Theorem \ref{thm:equiv_B_strong} shows that, in order to prove the implication \eqref{eq:thm}-\eqref{sign_p} $\Rightarrow$ \eqref{eq:vi}, it is indispensable that \eqref{des} is true. In view of Lemma \ref{sign_b}, this is equivalent to 
\begin{subequations}  \label{needed} \begin{gather}
p \leq 0 \text{ a.e.\,in } D_n^{convex},
\\p \geq 0 \text{ a.e.\,in } D_n^{concave}.
   \end{gather}\end{subequations}

 \ro{If $\mu(\AA)=0,$ the assertions in Theorem \ref{thm:equiv_B_strong} are 
valid  for all representatives and there is no need to impose \eqref{cqm}. Recall that, according to Theorem \ref{thm}, \eqref{sign_p0}--\eqref{sign_p1} is replaced by \eqref{sign_p'} if  $\mu(\AA)=0.$}
   
   When establishing strong stationary optimality conditions for an optimal control problem with non-smooth PDE - and control constraints,
it is common to require that the set where the non-smoothness and control constraints are active at the same time has measure zero \cite[Assump.\,3.3]{mcrf}. See also \cite{wachsm_2014}, for the  case when the control constrained  minimization problem is governed by a variational inequality.

\ro{We emphasize that  Assumption \ref{assu:cq} is  less restrictive than \cite[Assump.\,3.3]{mcrf}, since  the desired sign condition for the adjoint state, cf.\,\eqref{needed}, is missing only for a.a.\, $x  \in \clos \AA \setminus \AA$ for which $\beta$ is non-smooth and convex at $\yy(x)$, see \eqref{sign_p0} and \eqref{sign_p}; note that the latter is as a direct consequence of \eqref{grad_f0}. 
Cf.\,also Remark \ref{rem:eps}.}
\end{remark}

\section{A condition on the given data for which Assumption \ref{assu:h_a} is satisfied}\label{sec:rem}

In this last section we provide a condition which guarantees the validity of the "constraint qualification" regarding the non-linear term acting on the control in the state equation, i.e.,
Assumption \ref{assu:h_a}.
Since in the recent contribution \cite{p3} we   work again  in the framework of  \cite{p1}, it makes sense to check if  the requirements from \cite{p1} (which were not needed in the present manuscript) help us prove that Assumption \ref{assu:h_a} is true for $\e$ small enough.
As the next result shows, this is possible, if we slightly modify \cite[Assump.\,4.11]{p1}. Cf.\,also Remark \ref{remm} below.

\begin{lemma}\label{lem}
\ro{Suppose that 
\begin{equation}\label{f0}
 \esssup_{\ro{x \in  E}} f(x)   \leq \beta(0), \quad \esssup_{\ro{x \in D\setminus \bar E}} f(x) <\beta(0),\end{equation}
Then, for each   sequence $\{ g_\e\}\subset \ld \cap \lde$ that is uniformly bounded in $\lde$ and satisfies $g_\e \leq 0 \ \ae E$, there exists $\e_0>0$, so that }
\[S_\e( g_\e) \leq 0 \quad \ae \ro{D},\quad \forall\,\e\in (0,\e_0].\]
Note that the control to state map from Lemma \ref{lem:S} is denoted here by $S_\e.$
\end{lemma}
\begin{proof}As $\{g_\e\}$ is  uniformly bounded in $\lde$, by assumption, there exists $c>0,$ independent of $\e,$ so that 
\[f+\e g_\e \leq f+\e c \leq \beta(0) \quad \ae D \setminus E,\]
for $\e>0$ small enough. Note that the second inequality is a consequence of \eqref{f0}.
\ro{Further, from $g_\e \leq 0 \ \ae E$, by assumption, and \eqref{f0}, we deduce 
\[f+\e g_\e  \leq \beta(0) \quad \ae D .\]
By testing  \eqref{eq} associated to $g_\e$ with $\max\{0,S_\e(g_\e)\}$ and by employing the monotonicity of $\beta$ (Assumption \ref{assu:stand}), we see that }
 \[S_\e(g_\e) \leq 0 \quad \ae D\] for $\e>0$ small enough. The proof is now complete.
\end{proof}

\begin{corollary}\label{corr}
Let $\{\bar g_\e\}$ be a sequence of  local minimizers to \eqref{p10} that converges in $\WW$. 
If \eqref{f0} is true, $\ro{N=2},$ and $s>1$, then Assumption \ref{assu:h_a} is satisfied for $\e>0$ small enough.
\end{corollary}
\begin{proof}
Since $\ro{N=2}$ and $s>1$, we have the embedding $\WW \embed \ld \cap \lde,$ see e.g.\,\cite[p.\,88]{kaballo}.
Then, the first assertion in Assumption \ref{assu:h_a} follows from  $H_\e' \geq 0$ (see \eqref{reg_h'})  and Lemma \ref{lem}. Note that these imply 
$$ \frac{1}{\eps}H_\eps'(\bar g_\e)\bar y_\e \leq 0<\eps \quad \ae D \setminus \bar E,$$for $\e>0$ small enough. 
Moreover,
\[|\frac{1}{\eps}H_\eps'(\bar g_\e)\bar y_\e-\eps|=\eps-\frac{1}{\eps}H_\eps'(\bar g_\e)\bar y_\e \geq \eps \quad \ae D \setminus \bar E,\]which yields the second statement in Assumption \ref{assu:h_a}.
\end{proof}

\begin{remark}\label{remm}
In \cite[Thm.\,5.2]{p1} it was shown that, given a locally optimal (shape) function $\bar g_{\mathfrak{sh}}$, there exists a sequence of local minimizers of \eqref{p10} that converges to $\bar g_{\mathfrak{sh}}$ in the space  $\WW$. \ro{The analysis therein requires working in two dimensions only \cite[Rem.\,2.7]{p1}. Moreover, it was necessary that}  $s>1$ \cite[Rem.\,4.4]{p1} and that \cite[Assump.\,4.11]{p1} holds true.
In \cite{p3} we  find ourselves again in the framework of \cite{p1} and since \eqref{f0} is  just a slightly modification 
of the second alternative requirement in   \cite[Assump.\,4.11]{p1}, the hypotheses in Corollary \ref{corr} seem to be reasonable when one thinks at the subsequent work \cite{p3}. \\\ro{Since the second alternative requirement in   \cite[Assump.\,4.11]{p1} involves also the condition $y_d\geq 0 \, \ae E$, one can show by a comparison argument combined with $S_\eps(g_\eps) \leq 0$ (Lemma \ref{lem}) and the non-negativity of $\zeta $ (cf.\,Lemma \ref{sign_b} and \eqref{clarke0})  that the approximating adjoint state in the framework of \cite{p3} satisfies $p_\e \leq 0 \,\ae D$ (for $\e$ small), see \eqref{adj}. That is, in the case that $\beta$ is convex around its non-differentiable points, the desired sign condition for the adjoint state  (Remark \ref{needed}) is a direct consequence of the indispensable requirement \cite[Assump.\,4.11]{p1} and  \eqref{f0}, provided that $\e$ is small enough.} \end{remark}

\section*{Acknowledgment}
This work was supported by the DFG grant BE 7178/3-1 for the project "Optimal Control of Viscous
Fatigue Damage Models for Brittle Materials: Optimality Systems".

\appendix
\section{Proof of Lemma \ref{lem:tool}}
\label{sec:a}

Although many of the arguments are well-known  \cite{barbu84}, we give a detailed proof, for completeness and for convenience of the reader.
{We define a control problem \eqref{eq:q_eps} governed by a  state equation where the non-differentiable function ${\beta}$  is replaced by   a smooth approximation (step (I)).
Then, by arguments inspired by e.g.\ \cite{barbu84}, it follows that $\bar g$ can be approximated by a sequence of local minimizers of \eqref{eq:q_eps} (step (II)). Passing to the limit in the adjoint system associated to the regularized optimal control problem \eqref{eq:q_eps} finally yields the desired assertion (step (III)). 
\\(I) We denote by $\overline{ B_{L^2(D)}(\bar g,r})$, $r>0,$ the (closed) ball of local optimality of $\bar g \in \FF$. For $\gamma>0$ arbitrary, but fixed, small, consider the smooth  optimal control problem
\begin{equation}\tag{$P_{\gamma}$}\label{eq:q_eps}
 \left.
 \begin{aligned}
  \min_{g \in \FF\cap \clos{B_{\ld}(\bar g,r)}} \quad & J(y,g)+\frac{1}{2}\|g-\bar g\|^2_{\WW}\\   \text{s.t.} \quad & 
  \begin{aligned}[t]
   -\laplace y + \beta_{\gamma}(y)+\frac{1}{\eps}H_\eps(g)y&=f+\eps g \quad \text{ in }D,
   \\y&=0 \quad \text{on } \partial D, \end{aligned} 
 \end{aligned}
 \quad \right\}
\end{equation}where $$J(y,g):=\int_E (y(x)-y_d(x))^2 \; dx  +\alpha\,\int_{D} (1-H_\eps(g))(x)\,dx+\frac{1}{2}\|g-\bar g_{\mathfrak{sh}}\|^2_{\WW},$$
 and $\beta_\gamma:\R \to \R$ is defined as \begin{equation}\label{b_g}
\beta_\g (v):=\int_{-\infty}^{\infty} \beta(v-\g s)\psi(s)\,ds ,
\end{equation}
 where $\psi \in C_c^\infty(\mathbb{R}),\ \psi \geq 0, \ \supp \psi \subset [-1,1]$ and $\int_{-\infty}^{\infty} \psi(s)\,ds=1$.
As a consequence of the Lipschitz continuity of $\beta$ (Assumption \ref{assu:stand}.\ref{it:stand2}) we have
 \ro{ \begin{equation}\label{eps}\begin{aligned}
|\beta_\g(v) -\beta(v)| &\leq\int_{-1}^{1} |\beta(v-\g s)-\beta(v)|\psi(s)\,ds 
\\&\leq \g L_{M+1} \int_{-1}^{1}  |s|\psi(s)\,ds\leq \g L_{M+1}  \quad \forall\,v \in [-M,M],\  \forall\,\g\in (0,1],
 \end{aligned} \end{equation}}where $M>0$ is arbitrary but fixed. Moreover, $\beta_\g$ is Lipschitz continuous in the same sense as $\beta$, with  Lipschitz constant $L_{M+1}$.

Now we turn our attention to the investigation of the state equation in \eqref{eq:q_eps}.
First and foremost, we observe that we can associate a control-to-state map $S_\gamma:\ld \to H_0^1(D) \cap H^2(D)$, since $\beta_\gamma$ has the same properties as its non-smooth counterpart $\beta$, cf.\,Assumption \ref{assu:stand}. \ro{In particular, we have
\begin{equation}
\|S_\gamma(g)\|_{\li} \leq   \ro{C\|f-\beta_\g(0)\|_{\ld}+c\,\eps\,\|g\|_{\ld},}
\quad \forall\,g \in \ld,
\end{equation}where $C>0$ is independent of $\g, \e, f$ and $g$, see \eqref{est_y}. As a consequence of \eqref{eps}, it holds 
\begin{equation}\label{byg}
\|S_\gamma(g)\|_{\li} \leq   c_1+c_2\,\|g\|_{\ld}
\quad \forall\,g \in \ld,\ \forall\,\g \in (0,1],
\end{equation}where $c_1,c_2>0$ are independent of $\g$. In addition, 
 \begin{equation}\label{s_conv}
  S_\gamma(g_{\gamma})-S(g_{\gamma}) \to 0  \quad \text{in } H_0^1(D) \  \ \text{as } \gamma \to 0
  \end{equation}for each uniformly bounded sequence $\{g_\g\}$ in $\ld$. To see this, one subtracts the equation solved by  $S_\gamma(g_{\gamma})$ from the one solved by $S(g_{\gamma}) $ for $\g>0$ fixed, small. Using the monotonicity of $\beta_\g$, $H_\eps\geq 0,$ and \eqref{eps}, we arrive at 
  \[\| S_\gamma(g_{\gamma})-S(g_{\gamma})\|_{H_0^1(D)} \leq \|\beta_\g(S(g_{\gamma})) -\beta(S(g_{\gamma}))\|_{\ld}\leq c\, \g L_{M+1}, \quad \forall\,\g \in (0,1], \]
  where $c>0$ depends only on $D$ and $M:=
\|S(g_{\gamma})\|_{\li}$; note that $M>0$ is independent of $\g$, by \eqref{est_y},  and since $\{g_\g\}$ is uniformly bounded in $\ld$. This proves \eqref{s_conv}.}
As a consequence of the Lipschitz continuity of $S$ (Lemma \ref{lem:S}) we also have 
 \begin{equation}\label{s_conv2}
  S_\gamma(g_{\gamma}) \to S(g)  \quad \text{in } H_0^1(D), \text{ if }g_\g \to g \text{ in }\ld.  \end{equation}

Moreover, we remark that $S_\gamma:\ld \to H_0^1(D)$ is G\^ateaux-differentiable,  and that its derivative $u:=S_\gamma '(g)(h)$ at $g\in \ld$ in direction $h \in \ld$ is the unique solution of 
\begin{equation}\label{lin} 
  \begin{aligned}[t]
   -\laplace u + \beta_\gamma'(y)u+\frac{1}{\eps}H_\eps(g)u+\frac{1}{\eps}H_\eps'(g)yh&=\eps\, h \quad \text{in }D,
   \\u&=0 \quad \text{on } \partial D,\end{aligned} 
\end{equation}
 where  $y := S_\gamma(g)$. Further, by employing the direct method of the calculus of variations and the weak to weak continuity of $S_{\gamma}: \FF \to H_0^1(D) \cap H^1(D) $ as well as the weak to strong continuity of $H_\eps: \FF \to  \ld$ (see Lemma \ref{wc}), 
we see that \eqref{eq:q_eps} admits a global solution $g_\gamma \in \FF \cap \clos{B_{\ld}(\bar g,r)}$.  Note that the set $ \FF \cap \overline{ B_{L^2(D)}(\bar g,r)}$ is weakly closed, i.e., for $\{z_n\}\subset  \FF \cap \overline{ B_{L^2(D)}(\bar g,r)}$ we have 
\[z_n \weakly z \quad \text{in }\WW \Rightarrow z \in \FF\cap \overline{ B_{L^2(D)}(\bar g,r)}.\] 
  
(II)  For simplicity, we abbreviate in the following  \begin{subequations} \begin{gather}
j(g ):=J(S(g),g) \quad \forall\, g \in \WW,\label{jj}\\
j_{\gamma}(g):=J(S_\gamma(g),g)+\frac{1}{2}\|g-\bar g\|_{\WW}^2 \quad \forall\, g \in \WW. \label{jn}
\end{gather}\end{subequations} 
Due to \eqref{s_conv2},  it holds 
  \begin{equation}\label{j}
  j(\bar g)=\lim_{\g \to 0} J(S_{\g}(\bar g),\bar g)=\lim_{\g \to 0} j_\g(\bar g) \geq \limsup_{\g \to 0} j_\g(g_\g),  \end{equation}\normalsize where for the last inequality we relied on the fact that $g_\g$ is a global minimizer of \eqref{eq:q_eps} and that $\bar g$ is admissible for \eqref{eq:q_eps}. By the definition of $j_{\g}$, \eqref{j} can be continued as 
   \begin{equation}\label{j1}\begin{aligned}
  j(\bar g ) &\geq \limsup_{\g \to 0}  J(S_\gamma(g_\g),g_\g)+{\frac{1}{2}\|g_\g-\bar g\|_{\WW}^2}
  \\&\quad =\limsup_{\g \to 0}  J(S(g_\g),g_\g)+{\frac{1}{2}\|g_\g-\bar g\|_{\WW}^2}
   \\&\geq \liminf_{\g \to 0}  J(S(g_\g),g_\g)+{\frac{1}{2}\|g_\g-\bar g\|_{\WW}^2}
     \\&  \geq j(\bar g),
\end{aligned}  \end{equation}where the identity in \eqref{j1} is obtained from \eqref{s_conv} \ro{combined with $g_\g \in \overline{ B_{L^2(D)}(\bar g,r)}.$} This also implies the last estimate in \eqref{j1}. From \eqref{j1} we can now conclude \begin{equation}\label{conv_g}
g_\gamma \to \bar g \quad \text{in }\WW.\end{equation}
A localization argument then yields that for  $\g>0$ small enough $g_\gamma$ is  a local minimizer of $\min_{g \in \FF} j_\g(g)$. Further, we abbreviate $y_\g:=S_\g(g_\g)$ and remark that 
\begin{equation}\label{conv_y}
y_\gamma \to \bar y \quad \text{in }H_0^1(D) \cap H^2(D).\end{equation}
While the convergence in $H_0^1(D)$ is a consequence of \eqref{s_conv2}, the convergence in $H^2(D)$ can be proven as follows. First, \ro{by \eqref{byg} and $g_\g \in \overline{ B_{L^2(D)}(\bar g,r)}$}, there is a constant $C>0$, independent of $\g$, so that
\begin{equation}
\|y_\g\|_{\li} \leq C \ro{\quad \forall\,\g \in (0,1].}
\end{equation}
 As $\beta_\g$ is Lipschitz continuous in the same sense as $\beta$, this allows us to prove an estimate similar to \eqref{eq:flip}, namely
\begin{equation}\label{est_b}
   \|\beta_\g(y_\g) - \beta_\g(\bar y)\|_{L^2(D)} \leq L_C \, \|y_\g - \yy\|_{L^2(D)}\ro{\quad \forall\,\g \in (0,1].} 
  \end{equation}Then, by subtracting 
the equation associated to $\yy$ from the one associated to $y_\g$ and rearranging the terms, where one relies on the $C^{1,1}$ regularity of $D$, in combination with \eqref{conv_g}, the continuity of $H_\eps$, and \eqref{eps}, we finally obtain $
y_\gamma \to \bar y \  \text{in }H^2(D)$.

(III) We observe that there exists an adjoint state $p_\g \in \hd \cap H^2(D)$ so that \begin{equation}\label{adj_pp}
-\laplace p_\g+\beta_\g'(y_\g) p_\g+\frac{1}{\eps}H_\eps( g_\g)p_\g=2\chi_E( y_\g-y_d) \ \text{ in }D, \quad p_\g=0 \text{ on }\partial D.
\end{equation}   Now, let $ h \in \FF$ be arbitrary but fixed and  test \eqref{adj_pp} with $u_\g:=S_\g'(g_\g)(h-g_\g)$ and \eqref{lin} with $p_\g$. Then, since $g_\gamma$ is  a local minimizer of $\min_{g \in \FF} j_\g(g)$, we have   \begin{equation}\begin{aligned}\label{necc_g}
0&\leq (2\chi_E( y_\g-y_d),u_\g)_{L^2(D)}-(\alpha H_\eps'( g_\g),h-g_\g)_{\ld} +(2g_\g-\bar g_{\mathfrak{sh}}-\bar g,h-g_\g)_{\WW}
\\&=(-\laplace p_\g+\beta'( y_\g) p_\g+ \frac{1}{\eps}H_\eps( g_\g)p_\g,u_\g)_{L^2(D)}-(\alpha H_\eps'(g_\g),h-g_\g)_{\ld} 
\\&\quad +(2g_\g-\bar g_{\mathfrak{sh}}-\bar g,h-g_\g)_{\WW}
\\&=(-\laplace u_\g,p_\g)_{L^2(D)}+(\beta'( y_\g) u_\g ,p_\g)_{L^2(D)}+\frac{1}{\eps}(H_\eps(g_\g)u_\g,p_\g)_{L^2(D)}-(\alpha H_\eps'( g_\g),h-g_\g)_{\ld} 
\\&\quad +(2g_\g-\bar g_{\mathfrak{sh}}-\bar g,h-g_\g)_{\WW}
\\&=(\eps\, p_\g-\frac{1}{\eps}H_\eps'(g_\g)y_\g p_\g,h-g_\g)_{\ld}-(\alpha H_\eps'( g_\g),h-g_\g)_{\ld} 
\\&\quad+(2g_\g-\bar g_{\mathfrak{sh}}-\bar g,h-g_\g)_{\WW} \quad \forall\,h \in \FF.
\end{aligned} \end{equation}
Going back to \eqref{adj_pp}, we see that the monotonicity of $\beta_\gamma$ and the fact that $H_\eps (g_\g)\geq 0$, as well as \eqref{conv_y}, allow us to find uniform bounds for $\{p_\g\}$ so that we can select a subsequence which satisfies 
\begin{equation}\label{conv_p}
p_\g \weakly p \quad \text{in }\hd \cap H^2(D).
\end{equation}Note that here we also relied on
\begin{equation}\label{est_y_g}
\|\beta'(y_\g)\|_{\li} \leq C,
\end{equation}\ro{for $\g>0$ small enough,} with $C>0$ independent of $\g$. This is due to \eqref{conv_y} combined with the Lipschitz continuity of $\beta_\g$ on bounded sets, cf.\,\eqref{eq:flip}.

As a consequence of \eqref{conv_g}, \eqref{conv_y}, the compact embedding $\hd \cap H^2(D) \embed \embed \li$, and the continuity of $H_\eps',$  we can pass to the limit in \eqref{necc_g} and we arrive at \begin{equation}\begin{aligned}
0\leq (\eps\, p-\frac{1}{\eps}H_\eps'(\bar g)\bar y p,h-\bar g)_{\ld}-(\alpha H_\eps'( \bar g ),h-\bar g )_{\ld}+(\bar g-\bar g_{\mathfrak{sh}},h-\bar g)_{\WW}
\end{aligned} \end{equation}\ro{for all $h \in \FF.$}
 It remains to show \eqref{lem:adj}-\eqref{lem:clarke0}. To this end, we observe that, by \eqref{est_y_g}, there exist \ro{a not relabelled subsequence $\{\beta_\g'(y_\g)\}$} and  $\zeta \in \li$ so that 
\begin{equation}\label{conv_chi}
\beta_\g'(y_\g) \weakly^* \zeta \quad \text{in }\li.
\end{equation}Passing to the limit in \eqref{adj_pp}, where one relies on \eqref{conv_p}, the continuity of $H_\eps$, \eqref{conv_g} and \eqref{conv_y}, then implies \eqref{lem:adj}. 

Now, to show \eqref{lem:clarke0}, we use   arguments inspired by  the proofs of \cite[Prop.\,2.17, p.\,15]{mcrf} and \cite[Prop.\,3.10, p.\,16]{jnsao}. For completeness and for convenience of the reader, we recall them here.
Let $z$ be a non-differentiable point of $\beta$ for which $\beta$ is convex on $[z-\delta_z,z+\delta_z],$ where $\delta_z>0$ is given by Assumption \ref{assu:reg}. Then, $\beta_\g$ is convex on $[z-3\delta_z/4,z+3\delta_z/4]$ for each $\gamma\in [0,\delta_z/4]$; this can be proven by  calculations that involve the definition of convexity and \eqref{b_g}.  Now, define 
\begin{equation}\label{mz}
\NN_z:=\{x \in D: |\bar y(x)-z|  \leq \delta/2\}.
\end{equation}
Then
\begin{equation}\label{eq:mmm0}
 {\int_{ {\NN_z}} \beta_\g'(y_\g(x)) (\pm \tau) \varphi(x) \,dx \leq \int_{ {\NN_z}}  (\beta_\g(y_\g(x) \pm \tau)-\beta_\g(y_\g(x))) \varphi(x) \,dx}\end{equation}
for  $\tau \in \R$ small enough with $|\tau| \leq \delta/8$, $\varphi \in C^{\infty}_c(D)$ with $\varphi \geq 0$ and $\gamma>0$ small enough. 
Note that here we relied on the convergence \eqref{conv_y}, which implies that $|y_\g(x)-\yy(x)| \leq \delta /8 $ in $D$ for $\g>0$ small enough.
Passing to the limit $\g \searrow 0$ in \eqref{eq:mmm0} gives in turn 
\begin{equation*}
\int_{ {\NN_z}} \zeta(x)  (\pm \tau) \varphi(x) \,dx \leq \int_{ {\NN_z}}  (\beta(\yy(x) \pm \tau)-\beta(\yy(x))) \varphi(x) \,dx, \end{equation*}
where we employed   \eqref{conv_chi} and \eqref{est_b} combined with \eqref{eps}  and \eqref{conv_y}.
Now, the fundamental lemma of the calculus of variations yields 
$$
\frac{\beta(\yy(x))-\beta(\yy(x)-\tau)}{ \tau} \leq \zeta(x)   \leq \frac{\beta(\yy(x)+ \tau)-\beta(\yy(x))}{ \tau} \quad \ae{ {\NN_z}}
$$for all $\tau \in \R$ small enough with $0<|\tau| \leq \delta/8$. Letting $\tau \searrow 0$ results in 
\begin{equation}\label{eq:m} 
\zeta(x) \in [ \beta'_- (\yy(x)), \beta'_+ (\yy(x))] \quad  \ae { {\NN_z}}.
\end{equation}
In order to show that \eqref{eq:m} holds in the case that $\beta$ is locally concave around $z$, we argue in a similar way and obtain
$$
\frac{\beta(\yy(x))-\beta(\yy(x)-\tau)}{ \tau} \geq \zeta(x)   \geq \frac{\beta(\yy(x)+ \tau)-\beta(\yy(x))}{ \tau} \quad \ae{ {\NN_z}}
$$for all $\tau \in \R$ small enough with $0<|\tau| \leq \delta/8$. Thus, passing to the limit $\tau \searrow 0$ yields 
\begin{equation}
\zeta(x) \in [ \beta'_+ (\yy(x)), \beta'_- (\yy(x))] \quad  \ae { {\NN_z}}.
\end{equation} To summarize we have shown that 
\begin{equation}\label{eq:m1}
\zeta(x) \in [\min\{ \beta'_+ (\yy(x)), \beta'_- (\yy(x))\},\max\{ \beta'_+ (\yy(x)), \beta'_- (\yy(x))\}] \quad  \ae { \cup_{z \in \NN} \NN_z},
\end{equation}where $\NN$ denotes the set of non-smooth points of $\beta$.
Now, let $z_1,z_2 \in \NN$ be given such that $\beta$ is differentiable on $(z_1,z_2)$. \ro{Then,  $(z_1+\delta_{z_1}/2,z_2-\delta_{z_2}/2)\neq \emptyset$, as $[z-\delta_{z},z+\delta_{z}], z \in \NN,$ are disjoint, by Assumption \ref{assu:reg}}. Moreover, $\beta'$ is continuous on $(z_1+\delta_{z_1}/4,z_2-\delta_{z_2}/4).$ In view of \eqref{b_g}, it holds $\beta_\g=\beta \star \psi_\g$, where $\psi_\g(\cdot)=\psi(\cdot/\g)/\g$.  Thus, 
\begin{equation}\label{unif_b}
\beta_\g'=\beta' \star \psi_\g \rightrightarrows \beta' \text{ on } [z_1+3\delta_{z_1}/8,z_2-3\delta_{z_2}/8] \text{ as }\g \searrow 0,
\end{equation}\ro{where the symbol $\rightrightarrows$ stands for uniform convergence.}
Now, let $\MM_{z_1}^{z_2}:=\{x \in D: \bar y (x) \in (z_1+\delta_{z_1}/2,z_2-\delta_{z_2}/2)\}$. Then, for $\g>0$ small enough, $y_\g (x) \in (z_1+3\delta_{z_1}/8,z_2-3\delta_{z_2}/8)$ for all $x\in \MM_{z_1}^{z_2};$ this is a consequence of \eqref{conv_y}. By employing again \eqref{conv_y}, combined with the Lipschitz continuity of $\beta$, Lebesgue dominated convergence, as well as  the uniform convergence in \eqref{unif_b}, we  obtain $$\beta_\g'(y_\g) \to \beta'(\yy) \text{ in }L^2(\MM_{z_1}^{z_2}).$$In light of \eqref{conv_chi}, this means that 
$$\zeta(x)=\beta'(\yy)(x) \quad \text{a.e.\,in }\MM_{z_1}^{z_2}.$$
  If $z \in \NN$ is such that  $\beta$ is differentiable on  $(-\infty,z)$, we argue in the exact same way as above to show
\begin{equation}\begin{aligned}
\zeta
=\beta'(\yy) \quad \text{a.e.\,where }\yy \in[-\|\yy\|_{\li}, z-\delta_z/2).
\end{aligned}\end{equation}Note that  since $\{y_\g\}$ is uniformly bounded in $\li$ (see \eqref{conv_y}), all the above arguments apply. The same is true if $z \in \NN$ is such that  $\beta$ is differentiable on  $(z,\infty).$ From the above identities combined with \eqref{mz} and \eqref{eq:m1} we deduce that \eqref{lem:clarke0} is true.
  The proof  is now complete.

\bibliographystyle{plain}
\bibliography{strong_stat_coupled_pde}

\end{document}